\documentclass[12pt]{amsart}
\usepackage{amsmath,amssymb,amsthm}
\usepackage{color}

\DeclareMathOperator{\tr}{{{Tr}}}

\DeclareMathOperator{\Alt}{{{Alt}}}

\numberwithin{equation}{section}

\begin{document}
\newtheorem{thm}{Theorem}[section]
\newtheorem{lem}{Lemma}[section]
\newtheorem{crl}{Corollary}[section]
\newtheorem{rmk}{Remark}[section]
\newtheorem{nt}{Note}[section]
\newtheorem{Example}{Example}[section]

\title[Remarks on natural differential operators]{Remarks on natural differential operators with tensor fields }

\author[J. Jany\v ska]{Josef Jany\v ska}

\address{\ 
\newline
Department of Mathematics and Statistics, Masaryk University
\newline
Kotl\'a\v rsk\'a 2, 611 37 Brno, Czech Republic
\newline
e-mail: janyska@math.muni.cz}
  
\maketitle

\begin{abstract}
 We study natural differential operators transforming two tensor fields into a tensor field. First, it is proved that all bilinear operators are of order one, and then we give the full classification of such operators in several concrete situations. 
\end{abstract}

\bigskip

\noindent{
\\
{\ }
}

\section*{Introduction}

In differential geometry, many natural differential operators transforming two tensor fields into a tensor field are used. For instance, the Fr\"olicher-Nijenhuis bracket of two tangent-valued forms (see \cite{FroNij56}), the Shouten bracket of two multi-vector fields (see \cite{Sch53}), the Lie derivative of a form with respect to a tangent-valued
form (see \cite{KolMicSlo93}) and so on. The common property of all such operators is that they are $ \mathbb{R} $-bilinear and of order one.

In the present paper, we shall discuss such operators in the case that one of the input tensor fields $ \varphi $ is of type $ (1,p) $ and the second input tensor field $ \psi $ is of type $ (r,s) $. We shall prove that for $ p  > 1 $, $ s> r $, any natural differential operator $ \Phi $ transforming $ \varphi $ and $ \psi $ into a $ (r,s+p) $-tensor field is $ \mathbb{R} $-bilinear and of order one. If we assume that the operator is bilinear, then it is of order one for any $ p,\, r, \, s. $   Choice of the tensor field $ \varphi $ of type $ (1,p) $ is motivated by the paper \cite{YanAko68} where operators of the above type were studied under some special properties of the input fields. In addition to the result of \cite{YanAko68}, we give the full classification of operators without the assumption of special properties of the input fields.

We shall give as examples full classification of  natural bilinear operators transforming a vector field $ X $ or a (1,1)-tensor field $ \varphi $ or a (1,2)-tensor field $ S $ and a tensor field $ \psi $ into tensor fields. 

We assume that all operators are natural in the sense of \cite{KolMicSlo93}. We use the general properties of such operators.
To classify natural differential operators on tensor fields we use 
the method of an auxiliary linear symmetric connection $K$, \cite[p. 144]{KruJan90}, and the second-order reduction theorem, \cite[p. 165]{Sch54}. We assume that a $ k $-order natural operator also depends on a symmetric linear connection $ K $. Then, according to the second reduction theorem, such operator is factorized through the covariant derivatives up to the order $ k $ and covariant derivatives of the curvature tensor of $ K $ up to the order $ (k-2) $. Finally, we assume that the operator is independent of $K$.

\smallskip
All manifolds and mappings are assumed to be smooth.

\section{Preliminaries}
Let $ M $ be an $ m $-dimensional manifold and $ (x^{i}) $ local coordinates on $ M $. We shall denote as $ \partial_{i} $ and $ d^{i} $ local bases of vector fields and 1-forms. 

First of all, we shall discuss the order of natural operators transforming two tensor fields $ \varphi $ and $ \psi $ into tensor fields. We shall assume that $ \varphi $ is a tensor field of type $ (1,p) $.

\begin{thm}\label{Th: 1.1}
All finite order natural differential operators transforming a $(1,p)$, $p> 1$, tensor field $ \varphi $ and an $(r,s)$, $ s > r $, tensor field $ \psi $ into $(r,s+p)$ tensor fields $\Phi(\varphi,\psi)$ are $\mathbb{R}$-bilinear and of order 1.
\end{thm}

If we assume that the operator $\Phi$ is $\mathbb{R}$-bilinear we can consider weaker conditions on types of tensor fields $\varphi$ and $\psi$. 

\begin{thm}\label{Th: 1.2}
All finite order $\mathbb{R}$-bilinear natural differential operators transforming a $(1,p)$-tensor field $ \varphi $ and an $(r,s)$-tensor field $ \psi $ into $(r,s+p)$-tensor fields $\Phi(\varphi,\psi)$ are of order 1.
\end{thm}

\emph{Proof of Theorem \ref{Th: 1.1}:}
Let us assume a  $ k$-order, $ k\ge 1 $, natural differential operator
\begin{equation*}
\Phi: C^{\infty}(T^{(1,p)}M)\times C^{\infty}(T^{(r,s)}M)  \to C^{\infty}(T^{(r,s+p)}M)\,,
\end{equation*}
where $ p > 1 $ and $ s > r$.
Then the associated fibred morphism (denoted by the same symbol)
$$
\Phi: J^{k}(T^{(1,p)}M)\times_M J^{k}(T^{(r,s)}M)  \to T^{(r,s+p)}M
$$ 
is an equivariant mapping with respect to the actions of the $ (k+1) $-order differential group $ G^{k+1}_{m} = \mathrm{reg} J^{k}_{0}(\mathbb{R}^{m},\mathbb{R}^{m})_{0} $ on the standard fibres of  $ J^{k}(T^{(1,p)}M) $, $ J^{k}(T^{(r,s)}M) $ and $T^{(r,s+p)}M $. The restriction of the action of $ G^{k+1}_{m}  $ to constant multiples of the unite element of  $ G^{k+1}_{m} $ implies that $ \Phi $ has to satisfy the following condition
\begin{align}\label{Eq: 1.1}
k^{s-r+p} \Phi(j^{k}\varphi,j^{k}\psi) 
& =
\Phi(
k^{p-1}\varphi, k^{p}\partial\varphi,\dots,k^{p+k-1}\partial^{k}\varphi,
\\
&\quad\nonumber
k^{s-r}\psi,k^{s-r+1}\partial\psi,\dots,k^{s-r+k}\partial^{k}\psi)
\,
\end{align}
for all $ k \in \mathbb{R}^{+} $. 

All exponents in the equation \eqref{Eq: 1.1} are positive integers which implies, from the homogeneous function theorem (see \cite[p. 213]{KolMicSlo93}), that the operator $\Phi$ is a polynomial of orders $a_{l}$ in $ \partial^{l}\varphi $ and $ b_{l} $ in $ \partial^{l}\psi $ such that
\begin{equation}\label{Eq: 1.2}
\sum_{l=0}^{k} \big( (p+l-1)\, a_{l} + (s-r+l)\, b_{l}\big)
=
s-r+p\,. 
\end{equation}
Since all coefficients in \eqref{Eq: 1.2} are positive there are only two solutions in non-negative integers: a) $ a_{0} = 1\,\,\, b_{1} = 1 $ and the others $ a_{i},\,\, b_{i} $ are vanishing\,, b) $ a_{1} = 1\,, \,\, b_{0} = 1 $  and the others $ a_{i},\,\, b_{i} $ are vanishing. These solutions correspond to $\mathbb{R}$-bilinear 1st order operators.
\hfill$\square$  

\medskip
\emph{Proof of Theorem \ref{Th: 1.2}:} Let $p,\, r,\, s $ are arbitrary.
If we assume that the operator is $\mathbb{R}$-bilinear then it is a polynomial of orders $a_{l}$ in $ \partial^{l}\varphi $ and $ b_{l} $ in $ \partial^{l}\psi $ such that the equation \eqref{Eq: 1.2} is satisfied.
But now some coefficients in \eqref{Eq: 1.2} can be negative or vanishing. 
There are only two solutions in natural numbers which corresponds to $\mathbb{R}$-bilinear operators: a) $ a_{0} = 1\,\,\, b_{1} = 1 $ and the others $ a_{i},\,\, b_{i} $ are vanishing\,, b) $ a_{1} = 1\,, \,\, b_{0} = 1 $  and the others $ a_{i},\,\, b_{i} $ are vanishing. Hence all finite order natural $\mathbb{R}$-bilinear differential operators are of order 1.   
\hfill$\square$

\medskip
According to Theorems \ref{Th: 1.1} and \ref{Th: 1.2} all $\mathbb{R}$-bilinear natural differential operators $ \Phi $ are of the form
\begin{align}\label{Eq: 1.3}
\Phi(\varphi,\psi)
& = 
\bigg( A^{i_{1}\dots i_{r}m_{1}\dots m_{p}t_{1}\dots t_{s+1}}_{j_{1}\dots j_{s+p}kq_{1}\dots q_{r}} \, \varphi^{k}_{m_{1}\dots m_{p}}\, \partial_{t_{s+1}} \psi^{q_{1}\dots q_{r}}_{t_{1}\dots t_{s}} 
\\
& \quad\nonumber
+
 B^{i_{1}\dots i_{r}m_{1}\dots m_{p+1}t_{1}\dots t_{s}}_{j_{1}\dots j_{s+p}kq_{1}\dots q_{r}} \, \psi^{q_{1}\dots q_{r}}_{t_{1}\dots t_{s}} \, \partial_{m_{p+1}}\varphi^{k}_{m_{1}\dots m_{p}} \bigg) 
\\
&\qquad\nonumber
\partial_{i_{1}} \otimes \dots  \otimes \partial_{i_{r}} \otimes d^{j_{1}} \otimes \dots \otimes d^{j_{s+p}}\,,  
\end{align}
where $ A^{i_{1}\dots i_{r}m_{1}\dots m_{p}t_{1}\dots t_{s+1}}_{j_{1}\dots j_{s+p}kq_{1}\dots q_{r}} $
and $ B^{i_{1}\dots i_{r}m_{1}\dots m_{p+1}t_{1}\dots t_{s}}_{j_{1}\dots j_{s+p}kq_{1}\dots q_{r}}  $
are absolute invariant tensors (see \cite[p. 214]{KolMicSlo93}).
Such absolute invariant tensors are all possible linear combinations, with real coefficients, of tensor products of the identity $ \mathbb{I} $ of $ TM $, i.e. 
\begin{equation}\label{Eq: 1.4}
 A^{i_{1}\dots i_{r}m_{1}\dots m_{p}t_{1}\dots t_{s+1}}_{j_{1}\dots j_{s+p}kq_{1}\dots q_{r}}
 =
 \sum_{\sigma} a_{\sigma}\, \delta^{i_{1}}_{\sigma(j_{1})} \dots \delta^{t_{s+1}}_{\sigma(q_{r})}
\end{equation}
and
\begin{equation}\label{Eq: 1.5}
B^{i_{1}\dots i_{r}m_{1}\dots m_{p+1}t_{1}\dots t_{s}}_{j_{1}\dots j_{s+p}kq_{1}\dots q_{r}}
=
 \sum_{\sigma} b_{\sigma}\, \delta^{i_{1}}_{\sigma(j_{1})} \dots \delta^{t_{s}}_{\sigma(q_{r})}\,,
\end{equation}
$ a_{\sigma}\,,\,\, b_{\sigma} \in \mathbb{R} $, where $ \sigma $ runs all permutations of $ (r + s + p + 1) $ indices.

Moreover, to obtain natural operators, coefficients $ a_\sigma,\,\,\,b_\sigma $ have to satisfy some identities.
To calculate these identities, we use the method of an auxiliary linear symmetric connection $K$, \cite[p. 144]{KruJan90}, and the second reduction theorem, \cite[p. 165]{Sch54}. We assume that the operator $ \Phi $ also depends on $ K $. Then, by the second reduction theorem, the operator is factorized via the covariant derivatives of $ \varphi $ and $ \psi $ with respect to $ K $. So, we replace derivatives of tensor fields with covariant derivatives and assume that the operator is independent of $K$ which gives a system of homogeneous linear equations for $ a_{\sigma} $ and  $ b_{\sigma} $.

\section{Natural  $\mathbb{R}$-bilinear operators transforming  vector fields $ X $ and  tensor fields $ \psi $ into tensor fields of the same type as $ \psi $}

According to Theorem \ref{Th: 1.2} all such natural $\mathbb{R}$-bilinear operators are of order 1.

\subsection{Operator $ \Phi(X,-) $ applied to vector fields}
It is very well known that the Lie bracket is unique, up to a constant multiple, natural $\mathbb{R}$-bilinear operator transforming two vector fields into a vector field. We shall reprove this fact to demonstrate the method of an auxiliary linear symmetric connection.

\begin{thm}\label{Th: 2.1}
All natural $\mathbb{R}$-bilinear differential operators transforming two vector fields into vector fields are constant multiples of the Lie bracket.
\end{thm}

\begin{proof}
Let $ X$ and  $\psi = Y $ be vector fields. Then from \eqref{Eq: 1.3} - \eqref{Eq: 1.5}
$$
\Phi(X,Y) = \Phi^{i}\, \partial_{i}\,,
$$
where 
$$
\Phi^{i} = a_{1} \, X^{m} \, \partial_{m} Y^{i} + 
a_{2} \, X^{i} \, \partial_{m} Y^{m} + 
b_{1} \, Y^{m} \, \partial_{m} X^{i} +
b_{2} \, Y^{i} \, \partial_{m} X^{m}\,. 
$$
Let us assume a natural differential operator $ \Psi $ transforming vector fields $ X,\,\, Y $ and a linear symmetric connection $ K $ into vector fields. Then, according to the second reduction theorem, \cite[p. 165]{Sch54}, this operator factorizes through covariant derivatives  $ \nabla X $ and $ \nabla Y $ and it is an $\mathbb{R}$-bilinear operator. In coordinates we obtain
\begin{align*}
\Psi^{i} 
& = 
a_{1} \, X^{m} \, \nabla_{m} Y^{i} + 
a_{2} \, X^{i} \, \nabla_{m} Y^{m} + 
b_{1} \, Y^{m} \, \nabla_{m} X^{i} +
b_{2} \, Y^{i} \, \nabla_{m} X^{m}
\\
& = 
a_{1} \, X^{m} \, (\partial_{m} Y^{i} -  K_{m}{}^{i}{}_{p} Y^{p})
+ a_{2} \, X^{i} \, (\partial_{m} Y^{m} - K_{m}{}^{m}{}_{p} Y^{p})
\\
& \quad
+ b_{1} \, Y^{m} \, (\partial_{m} X^{i} - K_{m}{}^{i}{}_{p} X^{p})
+ b_{2} \, Y^{i} \, (\partial_{m} X^{m} - K_{m}{}^{m}{}_{p} X^{p})
\,,
\end{align*}
where $ K_{j}{}^{i}{}_{k} $ are the symbols of $ K $.
The part of $ \Psi^{i} $ independent of $ K $
coincides with $ \Phi^{i} $, so   
we obtain for $ \Phi^{i} $ the following
identity
$$
0
=
 (a_{1} \, X^{m} \, K_{m}{}^{i}{}_{p} 
+ a_{2} \, X^{i} \, K_{m}{}^{m}{}_{p})\, Y^{p}
+ (b_{1} \, Y^{m} \, K_{m}{}^{i}{}_{p}
+ b_{2} \, Y^{i} \, K_{m}{}^{m}{}_{p})\, X^{p}\,.
$$  
It is easy to see that this identity is satisfied if and only if
$$
a_{1} + b_{1} = 0\,, \quad a_{2} = 0 = b_{2}\,.
$$
So
$$
\Phi^{i} = a_{1} \, (X^{m} \, \partial_{m} Y^{i} - Y^{m} \, \partial_{m} X^{i})
= a_{1}\, [X,Y]^{i}\,
$$
and $ \Phi(X,Y)  $ is a constant multiple of the Lie bracket $ [X,Y] $.
\end{proof}

\subsection{Operator $ \Phi(X,-) $ applied to 1-forms}

\begin{thm}\label{Th: 2.2}
All natural $\mathbb{R}$-bilinear operators transforming a vector field $ X $ and a 1-form $ \psi $ into 1-forms are linear combinations, with real coefficients, of two operators
$$
d(\psi(X))\,, \quad i_Xd\psi\,.
$$ 
\end{thm}

\begin{proof}
Let $ X\,$ be a vector field and $\psi$ be a 1-form. Then by \eqref{Eq: 1.3} - \eqref{Eq: 1.5}
$$
\Phi(X,\psi) = \Phi_{i}\, d^{i}\,,
$$
where 
$$
\Phi_{i} = a_{1} \, X^{m} \, \partial_{m} \psi_{i} + 
a_{2} \, X^{m} \, \partial_{i} \psi_{m} + 
b_{1} \, \psi_{i} \, \partial_{m} X^{m} +
b_{2} \, \psi_{m} \, \partial_{i} X^{m}\,. 
$$
Now, we replace partial derivatives with covariant derivatives with respect to an auxiliary linear symmetric connection $ K $ and assume that the operator is independent of $ K $. We obtain the following
identity
$$
0
=
(a_{1} + a_{2} - b_{2})  \, X^{m} \, K_{m}{}^{p}{}_{i} \psi_{p}
- b_{1} \, \psi_{i} \, K_{m}{}^{m}{}_{p} X^{p}
\,.
$$  
So, we have
$$
a_{1} + a_{2} - b_{2} = 0\,, \quad b_{1} = 0\,
$$
and
$$
\Phi_{i} = a_{1} \, X^{m} \,( \partial_{m} \psi_{i} -   \partial_{i} \psi_{m})
+ b_{2} \, (X^{m} \, \partial_{i} \psi_{m} + \psi_{m} \, \partial_{i} X^{m})
\,
$$
which is the coordinate expression of  $a_{1} \, i_{X}d\psi + b_{2} \, d(\psi(X))$.
\end{proof}

\begin{rmk}\label{Rm: 2.1}
{\rm
In differential geometry the Lie derivative $ L_X\psi = i_Xd\psi + di_X\psi$ is very often used, but according to Theorem \ref{Th: 2.2}
any linear combination of $ d(\psi(X))\,, \,\, i_Xd\psi $ is a natural 1-form.  
}
\end{rmk}

\subsection{Operator $ \Phi(X,-) $ applied to (0,2)-tensor fields}

We assume a $ (0,2) $-tensor field $ \psi $.

\begin{thm}\label{Th: 2.3}
All natural $\mathbb{R}$-bilinear differential operators transforming a vector field $ X $ and a (0,2)-tensor field $ \psi $ into (0,2)-tensor fields are linear combinations, with real coefficients, of four operators
$$
L_{X}\psi\,, \quad L_X\widetilde\psi\,, \quad d(X\lrcorner \psi)\,,\quad d(X\lrcorner\widetilde{\psi})\,,
$$ 
where $ \widetilde{\psi} $ is the $ (0,2) $-tensor field given as $ \widetilde{\psi}(Y,Z) = \psi(Z,Y) $  and $ (X\lrcorner \psi)(Y) = \psi(X,Y) $ for any vector fields $ X, \,Y, \, Z $.
\end{thm}

\begin{proof}
Let $ X\,$ be a vector field and $\psi$ be a (0,2)-tensor field. Then by \eqref{Eq: 1.3} - \eqref{Eq: 1.5}
$$
\Phi(X,\psi) = \Phi_{ij}\, d^{i}\otimes d^{j}\,,
$$
where 
\begin{align*}
\Phi_{ij} 
& = 
a_{1} \, X^{m} \, \partial_{m} \psi_{ij} 
+ a_{2} \, X^{m} \, \partial_{m} \psi_{ji}
+ a_{3} \, X^{m} \, \partial_{i} \psi_{mj} 
+ a_{4} \, X^{m} \, \partial_{i} \psi_{jm}
\\
& \quad 
+ a_{5} \, X^{m} \, \partial_{j} \psi_{im} 
+ a_{6} \, X^{m} \, \partial_{j} \psi_{mi} 
+ b_{1} \, \psi_{ij} \, \partial_{m} X^{m} 
+ b_{2} \, \psi_{ji} \, \partial_{m} X^{m} 
\\
& \quad
+ b_{3} \, \psi_{mj} \, \partial_{i} X^{m}
+ b_{4} \, \psi_{jm} \, \partial_{i} X^{m}
+ b_{5} \, \psi_{im} \, \partial_{j} X^{m}
+ b_{6} \, \psi_{mi} \, \partial_{j} X^{m}
\,. 
\end{align*}

Now, we replace partial derivatives with covariant derivatives with respect to an auxiliary linear symmetric connection $ K $ and assume that the operator is independent of $ K $. We obtain the following
identity
\begin{align*}
0
&=
(a_{1} + a_{3} - b_{3}) \, X^{m} \, K_{m}{}^{p}{}_{i} \, \psi_{pj} 
+ (a_{2} + a_{4} - b_{4}) \, X^{m} \, K_{m}{}^{p}{}_{j} \, \psi_{pi}
\\
& \quad 
+ (a_{1} + a_{5} - b_{5}) \, X^{m} \, K_{m}{}^{p}{}_{j} \, \psi_{ip} 
+ (a_{2} + a_{6} - b_{6}) \, X^{m} \,  K_{m}{}^{p}{}_{i} \, \psi_{jp} 
\\
& \quad
+ (a_{3} + a_{6}) \, X^{m} \, K_{i}{}^{p}{}_{j} \, \psi_{mp} 
+ (a_{4} + a_{5}) \, X^{m} \, K_{i}{}^{p}{}_{j} \, \psi_{pm}
\\
& \quad
- b_{1} \, \psi_{ij} \, K_{m}{}^{m}{}_{p} X^{p}
- b_{2} \, \psi_{ji} \, K_{m}{}^{m}{}_{p} X^{p}
\,.
\end{align*} 
The above identity is satisfied if and only if 
$
b_{1}
 =
0
=
b_{2}
$ 
and the following system of homogeneous linear equations is satisfied 
\begin{align*}
a_{1} + a_{3} - b_{3}
& =
0\,, 
&
a_{1} + a_{5} - b_{5}
& =
0\,,
\\
a_{2} + a_{4} - b_{4}
& =
0\,,
&
a_{2} + a_{6} - b_{6}
& =
0\,,
\\
a_{4} + a_{5}
& =
0\,,
&
a_{3} + a_{6}
& =
0\,.
\end{align*}
Then we get
\begin{align*}
\Phi_{ij} 
& = 
a_{1} \, \big( X^{m} \, \partial_{m} \psi_{ij}
+  \psi_{mj} \, \partial_{i} X^{m}
+ \psi_{im} \, \partial_{j} X^{m}
\big)
\\
& \quad
+ a_{2} \, \big(
 X^{m} \, \partial_{m} \psi_{ji}
  + \psi_{mi} \, \partial_{j} X^{m}
 + \psi_{jm} \, \partial_{i} X^{m}
\big)
\\
& \quad
+ a_{3} \, \big(
X^{m} \, \partial_{i} \psi_{mj}
+ \psi_{mj} \, \partial_{i} X^{m}
- X^{m} \, \partial_{j} \psi_{mi}
- \psi_{mi} \, \partial_{j} X^{m}
\big)
\\
& \quad
+ a_{4} \, \big(
X^{m} \, \partial_{i} \psi_{jm}
+ \psi_{jm} \, \partial_{i} X^{m}
- X^{m} \, \partial_{j} \psi_{im} 
- \psi_{im} \, \partial_{j} X^{m}
\big)
\end{align*}
which is the coordinate expression of a linear combination of 
$ L_{X}\psi\,,$ $L_X\widetilde\psi\,,$ $d(X\lrcorner \psi)\,,$ $d(X\lrcorner\widetilde{\psi})\,. $
\end{proof}

\begin{rmk}
{\rm
Let us note that in above Theorem \ref{Th: 2.3} we have used the Lie derivation of any $ (0,2) $-tensor field defined as
\begin{equation*}
(L_X \psi)(Y,Z) = X.\psi(Y,Z) - \psi([X,Y],Z) - \psi(Y,[X,Z]) ,
\end{equation*}
for any vector fields $X, \,Y,\,Z$. In the case that $\psi$ is a 2-form this Lie derivative coincides with $L_X \psi = i_X d\psi + di_X\psi$. 
}
\end{rmk}

\section{Natural $\mathbb{R}$-bilinear operators transforming  (1,1)-tensor fields $ \varphi $ and  $(*,*)$-tensor fields $ \psi $ into $(*,*+1)$-tensor fields}

A (1,1) tensor field $ \varphi $ can be considered as a linear mapping
$ \varphi : TM \to TM $. As $ \tr\varphi  $ we assume the contraction and 
$ \mathbb{I}:TM \to TM $ is the identity. We do not assume special properties of $ \varphi $.

\subsection{Operator $ \Phi(\varphi,-) $ applied to (1,1)-tensor fields}

Full classification of natural $\mathbb{R}$-bilinear operators transforming two (1,1)-tensor fields into (1,2)-tensor fields was done in \cite [p. 152]{KruJan90} by using the other method. We recall this classification.   

\begin{thm}\label{Th: 3.1}
All natural $\mathbb{R}$-bilinear differential operators transforming  (1,1)-tensor fields $ \varphi $ and $ \psi $ into (1,2)-tensor fields form a 15 parameter family of operators given as a linear combination of the following operators
\begin{gather*}
d(\tr  \varphi) \otimes \psi\,,\quad 
  \psi \otimes d(\tr  \varphi)\,,\quad
d(\tr  \psi) \otimes \varphi\,,\quad 
 \varphi \otimes d(\tr  \psi) \,,\quad
 \\
(\tr \psi) \, d(\tr\varphi) \otimes \mathbb{I}\,,\quad
(\tr \psi) \, \mathbb{I} \otimes d(\tr\varphi) \,, \quad
(\tr \varphi) \, d(\tr\psi) \otimes \mathbb{I}\,,
\\ 
(\tr \varphi) \, \mathbb{I} \otimes d(\tr\psi) \,,\quad
(d(\tr\varphi)\circ\psi)\otimes \mathbb{I} \,, \quad
\mathbb{I} \otimes (d(\tr\varphi) \circ \psi)\,, 
\\ 
(d(\tr\psi)\circ\varphi ) \otimes \mathbb{I} \,, \quad
\mathbb{I} \otimes (d(\tr\psi) \circ \varphi)\,, \quad
d\big(\tr (\varphi\circ\psi)\big) \otimes \mathbb{I}\,,
\\ 
\mathbb{I} \otimes d\big(\tr (\varphi\circ\psi)\big)\,,\quad
N(\varphi,\psi)\,,
\end{gather*}
where $ \mathbb{I} $ is the identity of $ TM $ and $N(\varphi,\psi)$ is the Fr\"olicher-Nijenhuis bracket.
\hfill$\square$
\end{thm}

\begin{rmk}\label{Rm: 3.1}
{\rm
It is very well known that the  Fr\"olicher-Nijenhuis bracket, \cite{FroNij56}, has values in tangent-valued forms. If we assume operators transforming $ \varphi $ and $ \psi $ into tangent-valued 2-forms we obtain 8 parameter family generated by
\begin{gather*}
d(\tr  \varphi) \wedge \psi\,,\quad 
d(\tr  \psi) \wedge \varphi\,.\quad 
\\
(\tr \psi) \, d(\tr\varphi) \wedge \mathbb{I}\,,\quad
(\tr \varphi) \, d(\tr\psi) \wedge \mathbb{I}\,,\quad
\\ 
(d(\tr\varphi)\circ\psi)\wedge \mathbb{I} \,, \quad
(d(\tr\psi)\circ\varphi ) \wedge \mathbb{I} \,, \quad 
\\ 
d\big(\tr (\varphi\circ\psi)\big) \wedge \mathbb{I}\,, \quad N(\varphi,\psi)\,.
\end{gather*} 
}
\end{rmk}

\subsection{Operator $ \Phi(\varphi,-) $ applied to 1-forms}

\begin{lem}\label{Lm: 3.1}
We have the following 6 canonical 1st order natural $\mathbb{R}$-bilinear differential operators
\begin{gather*}
(\tr \varphi) \, d\psi\,,\quad \psi\otimes d(\tr \varphi) \,,\quad d(\tr \varphi) \otimes \psi\,,
\\
  d\psi \circ_{1} \varphi\,,\quad  d\psi \circ_{2} \varphi \,,\quad  d(\psi \circ \varphi)\,,
\end{gather*}
where
$ (d\psi \circ_{1} \varphi)(X,Y) = d\psi(\varphi(X),Y) $ and $  (d\psi \circ_{2} \varphi)(X,Y) = d\psi(X,\varphi(Y)) $ for any vector fields $ X, \,Y $.
\hfill$\square$
\end{lem}

\begin{rmk}\label{Rm: 3.2}
{\rm
We have the following independent operators with values in 2-forms
\begin{align*}
(\tr \varphi) \, d\psi\,,\quad \psi\wedge d(\tr \varphi) \,,\quad  \Alt(d\psi \circ_{1} \varphi) \,,\quad  d(\psi \circ \varphi)
\end{align*}
which follows from $ \Alt(d\psi \circ_{1} \varphi) = \Alt(d\psi \circ_{2} \varphi)  $\,, where $ \Alt $ is the antisymmetrisation.
}
\end{rmk}

\begin{thm}\label{Th: 3.2}
All natural $\mathbb{R}$-bilinear differential operators transforming $\varphi$ and $\psi$ into a (0,2) tensor fields form a six parameter family of operators which is a linear combination of operators from Lemma \ref{Lm: 3.1}. 
\end{thm}

\begin{proof}
According to 
\eqref{Eq: 1.3} -- \eqref{Eq: 1.5}
\begin{align*}
\Phi(\varphi,\psi) = \Phi_{ij}
\, d^{i} \otimes d^{j}\,,  
\end{align*}
where 
\begin{align*}
\Phi_{ij} 
&=
a_1\, \varphi^m_m \, \partial_i\psi_{j}
+ a_2\, \varphi^m_m \, \partial_j\psi_{i}
+ a_3\, \varphi^m_i \, \partial_m\psi_{j}
+ a_4\, \varphi^m_i \, \partial_j\psi_{m}
\\
& \qquad
+ a_5\, \varphi^m_j \, \partial_m\psi_{i}
+ a_6\, \varphi^m_j \, \partial_i\psi_{m}
\\
& \quad
+ b_{1}\, \psi_{i} \, \partial_m\varphi^m_j 
+ b_{2}\, \psi_{i} \, \partial_j\varphi^m_m
+ b_{3}\, \psi_{j} \, \partial_m\varphi^m_i 
+ b_{4}\, \psi_{j} \, \partial_i\varphi^m_m
\\
& \qquad
+ b_{5}\, \psi_{m} \, \partial_i\varphi^m_j 
+ b_{6}\, \psi_{m} \, \partial_j\varphi^m_i\,.
\end{align*}
In order to calculate relations for coefficients $ a_i,\,b_i $, $i=1,\dots,6$, we use the method of an auxiliary linear symmetric connection $K$, \cite[p. 144]{KruJan90}.
We replace derivatives of tensor fields with covariant derivatives and assume that the operator is independent of $K$. Then we get
\begin{align*}
0
& =
\varphi^m_m \, \big( a_1 + a_2 \big)\, K_i{}^p{}_j \,\psi_{p} 
\\
& \quad
+  \varphi^m_i \, \big[\big( a_3 + a_4 - b_{6}\big)\, K_m{}^p{}_j \,\psi_{p} 
- b_3\,  K_p{}^p{}_m \,\psi_{j} 
\big]
\\
& \quad
+  \varphi^m_j \, \big[\big( a_5 + a_6 - b_{5}\big)\, K_m{}^p{}_i \,\psi_{p} 
- b_1\,  K_p{}^p{}_m \,\psi_{i} 
\big]
\\
& \quad
+  \varphi^m_p \, \big[ b_1\,  K_m{}^p{}_j \,\psi_{i}
+ b_3\,  K_m{}^p{}_i \,\psi_{j}
+ (b_5 + b_6)\,  K_i{}^p{}_j \,\psi_{m}
\big]
\,.
\end{align*}
Then $b_2$ and $b_4$ are arbitrary, $ b_1 = b_3 = 0 $ and
\begin{align*}
b_6
& =
- b_5
\,, \qquad
a_2
 =
- a_1
\,, \qquad
a_4
 =
- a_3 - b_5
\,, \qquad
a_6
 =
- a_5 + b_5\,.
\end{align*}
Hence
\begin{align*}
\Phi_{ij}
&=
 a_1\, \varphi^m_m \, \big(\partial_i\psi_{j}
- \partial_j\psi_{i} \big)
\\
& \quad 
+ a_3\, \varphi^m_i \, \big(\partial_m\psi_{j}
- \partial_j\psi_{m} \big)
+ a_5\, \varphi^m_j \, \big(\partial_m\psi_{i}
- \partial_i\psi_{m} \big)
\\
& \quad
+ b_{2}\, \psi_{i} \, \partial_j\varphi^m_m 
+ b_{4}\, \psi_{j} \, \partial_i\varphi^m_m
\\
& \quad 
+
b_5 \, \big(
\varphi^m_j \, \partial_i\psi_m - \varphi^m_i \, \partial_j\psi_m
+ \psi_m \, (\partial_i\varphi^m_j -  \partial_j\varphi^m_i)
\big)\,.
\end{align*}
which is the coordinate expression of a linear combination of operators from Lemma \ref{Lm: 3.1}.
\end{proof}

\begin{crl}\label{Cr: 3.1}
If the 1-form $ \psi $ is closed, then all $\mathbb{R}$-bilinear 1st order natural differential operators form the 3-parameter family of operators generated by
$$ 
\psi \otimes d(\tr \varphi)\,,\quad
d(\tr \varphi) \otimes \psi\,, \quad
 d(\psi\circ\varphi)\,.
$$ 
Moreover, we have 2 independent operators $ \psi \wedge d(\tr \varphi) $ and $ d(\psi\circ\varphi) $ with values in 2-forms.
\end{crl}

\begin{rmk}\label{Rm: 3.3} 
{\rm
We can define others natural $\mathbb{R}$-bilinear operators on $\varphi$ and $\psi$. But, according to Theorem \ref{Th: 3.2}, they have to be obtained as linear
combinations of operators from Lemma \ref{Lm: 3.1}.

Let $ X,Y $ be vector fields,
in \cite{YanAko68} the operator $ \Phi  $ was defined as follows
\begin{align*}
\Phi(\varphi,\psi)(X,Y) 
& =(L_{\varphi(X)}\psi  - L_{X}(\psi\circ\varphi))(Y)
\end{align*}
which can be expressed as the linear combination of operators from Lemma \ref{Lm: 3.1}
$$
 \Phi (\varphi,\psi)
=
d\psi\circ_{1}\varphi - d(\psi\circ\varphi)\,.
$$

Further, according to \cite[p. 69]{KolMicSlo93}, we can define the Lie derivative of $ \psi $ with respect to $ \varphi $ as
\begin{equation*}
L_{\varphi}\psi = [i_{\varphi},d]\psi = i_{\varphi} d\psi - d i_{\varphi}\psi\,
\end{equation*}
which is a 2-form. It is easy to see that
\begin{equation*}
L_{\varphi}\psi = d\psi \circ_{1} \varphi + 
d\psi \circ_{2} \varphi - d(\psi \circ\varphi)\,.
\end{equation*}

For the identity of $TM$ we have  
\begin{equation*}
L_{\mathbb{I}}(\psi \circ \varphi) = 
i_{\mathbb{I}} d(\psi \circ \varphi) - d i_{\mathbb{I}}(\psi \circ \varphi) =  d(\psi \circ\varphi)\,
\end{equation*}
and we obtain, \cite{YanAko68},
\begin{equation*}
2\, \Alt\Phi(\varphi,\psi) = L_{\varphi}\psi - L_{\mathbb{I}}(\psi \circ \varphi)\,.
\end{equation*}
}
\end{rmk}

\subsection{Operator $ \Phi(\varphi,-) $ applied to (0,2) tensor fields}

 Let us denote as
$ \Alt \psi $ the antisymmetric part of $\psi$, 
i.e. in coordinates
\begin{align*}
\Alt \psi 
& = 
\tfrac 12\,(\psi_{ij} - \psi_{ji})\, d^{i} \otimes d^{j} 
= \psi_{ij}\, d^{i} \wedge d^{j} \,.
\end{align*}

First of all, we describe several types of 1st order natural $\mathbb{R}$-bilinear operators which are given by the tensorial operations (permutation of indices, tensor product, contraction, exterior differential).

\begin{lem}\label{Lm: 3.2}
$\psi\otimes d(\tr\varphi)$ defines six independent natural $\mathbb{R}$-bilinear differential operators
given by permutations of subindices, so for vector fields $ X, Y, Z $ we have operators
\begin{gather*}
 \psi(X,Y) \, d(\tr\varphi)(Z)\,,\quad \psi(Y,X) \, d(\tr\varphi)(Z)
\,,\quad \psi(X,Z) \, d(\tr\varphi)(Y)\,,
\\
 \psi(Z,X) \, d(\tr\varphi)(Y)
\,,\quad \psi(Y,Z) \, d(\tr\varphi)(X) \,, \quad \psi(Z,Y) \, d(\tr\varphi)(X)\,.
\end{gather*}
 
 Moreover, $ \Alt\psi \wedge d(\tr\varphi) $ is the unique operator with values in 3-foms.
\end{lem}

\begin{crl}\label{Cr: 3.2}
If the tensor field $ \psi $ is symmetric or antisymmetric then we get three independent operators from  Lemma \ref{Lm: 3.2}
\begin{align*}
\psi(X,Y) \, d(\tr\varphi)(Z)
\,,\,\,
\psi(X,Z) \, d(\tr\varphi)(Y)
\,,\,\,
\psi(Y,Z) \, d(\tr\varphi)(X))\,.
\end{align*}
\end{crl}

\begin{lem}\label{Lm: 3.3}
We have the following six independent natural $\mathbb{R}$-bilinear differential operators
\begin{gather*}
(\tr\varphi) \, d(\Alt \psi)  \,,\quad 
d(\Alt \psi) \circ_{1} \varphi\,,\quad 
d(\Alt \psi) \circ_{2} \varphi\,,\quad 
d(\Alt \psi) \circ_{3} \varphi\,,\quad 
\\
d(\Alt(\psi\circ_{1}\varphi))\,,\quad
d(\Alt(\psi\circ_{2}\varphi))\,.
\end{gather*}
\end{lem}

\begin{crl}\label{Cr: 3.3}
1. If $ \psi $ is symmetric then $ \Alt\psi = 0 $ and $ \Alt(\psi  \circ_{1}\varphi) = -  \Alt(\psi  \circ_{2}\varphi) $ and we have the unique operator from Lemma \ref{Lm: 3.3}
$$  
d(\Alt(\psi  \circ_{1}\varphi))\,.
$$

2. If $ \psi $ is antisymmetric then $ \Alt\psi = \psi $  and $ \Alt(\psi  \circ_{1}\varphi) = \Alt(\psi  \circ_{2}\varphi) $ and we have five independent operators from Lemma \ref{Lm: 3.3} 
\begin{gather*}
(\tr\varphi) \, d\psi \,,\quad 
d\psi \circ_{1} \varphi   \,,\quad 
d\psi \circ_{2} \varphi\,,\quad 
d\psi \circ_{3} \varphi\,,\quad
\\
d(\Alt(\psi\circ_{1}\varphi)) \,.
\end{gather*}
From $ \Alt(d\psi\circ_1\varphi ) = \Alt(d\psi\circ_2\varphi ) = \Alt(d\psi\circ_3\varphi )  $ we have 3 operators with values in 3-forms.

Moreover, if $ \psi $ is a closed 2-form, then there is the unique  operator from Lemma \ref{Lm: 3.3}
$$
d(\Alt(\psi  \circ_{1} \varphi))\,
$$
which has values in 3-forms.
\end{crl}

If the tensor fields $ \varphi $ and $ \psi $ satisfy 
$$
(\psi\circ_{1}\varphi)(X,Y)
=
(\psi\circ_{2}\varphi)(X,Y) \quad \Leftrightarrow \quad
\psi(\varphi(X),Y) = \psi(X,\varphi(Y))
$$
then $ \psi $ is said to be \emph{pure} with respect to $ \varphi $.
Natural $\mathbb{R}$-bilinear differential operators $ \Phi(\varphi,\psi) $ on pure tensor fields were studied in  \cite{Sal10, YanAko68}. We recall the main result.

\begin{thm}\label{Th: 3.3}
Let $\psi$ is pure with respect to $\varphi$.
Then
\begin{align*}
\Phi(\varphi,\psi)(X,Y,Z)
& =
\big(L_{\varphi(X)}\psi  - L_{X}(\psi\circ_{1}\varphi)\big)(Y,Z)
\end{align*}
is a (0,3)-tensor field.
\hfill$ \square $
\end{thm}

The above operator for pure tensor fields can be generalized for any tensor field $ \psi $.

\begin{thm}\label{Th: 3.4}
For any vector fields $ X,Y,Z $ the operators
\begin{align*}
\Phi_{1}(\varphi,\psi)(X,Y,Z)
& =
\big(L_{\varphi(X)}\psi  - L_{X}(\psi\circ_{1}\varphi)\big)(Y,Z)
\\
&\quad
- \big(L_{\varphi(Z)}\psi  - L_{Z}(\psi\circ_{1}\varphi)\big)(Y,X)
\\
 & \quad
 + (\psi\circ_{2}\varphi)(Y,[X,Z])  - (\psi\circ_{1}\varphi)(Y,[X,Z])
\end{align*}
and
\begin{align*}
\Phi_{2}(\varphi,\psi)(X,Y,Z)
& =
\big(L_{\varphi(X)}\psi  - L_{X}(\psi\circ_{2}\varphi)\big)(Y,Z) 
\\
& \quad
- \big(L_{\varphi(Y)}\psi  - L_{Y}(\psi\circ_{2}\varphi)\big)(X,Z)
\\
& \quad
- (\psi\circ_{2}\varphi)([X,Y],Z)  + (\psi\circ_{1}\varphi)([X,Y],Z)
\end{align*}
are (0,3)-tensor fields with the coordinate expressions
\begin{align*}
\Phi_{1}(\varphi,\psi)(X, & Y,Z)
 =
\big(
\varphi^{m}_{i}\,  \partial_{m}\psi_{jk} 
+ \varphi^{m}_{j}\,(\partial_{k}\psi_{mi}  - \partial_{i}\psi_{mk})
- \varphi^{m}_{k}\,  \partial_{m}\psi_{ji} 
\\
&  
+ \psi_{mi} \, ( \partial_{k}\varphi^{m}_{j} - \partial_{j}\varphi^{m}_{k} )
+ \psi_{jm} \, (\partial_{k}\varphi^{m}_{i} - \partial_{i}\varphi^{m}_{k} )
\\
&
+ \psi_{mk} \, ( \partial_{j}\varphi^{m}_{i} -  \partial_{i}\varphi^{m}_{j} )
\big)\, X^{i}\, Y^{j} \, Z^{k}\,
\end{align*}
and
\begin{align*}
\Phi_{2}(\varphi,\psi)(X, & Y,Z)
 =
\big(
\varphi^{m}_{i}\,  \partial_{m}\psi_{jk} 
- \varphi^{m}_{j}\,  \partial_{m}\psi_{ik}
+ \varphi^{m}_{k}\, (\partial_{j}\psi_{im} - \partial_{i}\psi_{jm})
\\
&   
- \psi_{im} \, (\partial_{k}\varphi^{m}_{j} - \partial_{j}\varphi^{m}_{k} )
+ \psi_{jm} \, (\partial_{k}\varphi^{m}_{i} - \partial_{i}\varphi^{m}_{k} )
\\
&
+ \psi_{mk} \, (\partial_{j}\varphi^{m}_{i} - \partial_{i}\varphi^{m}_{j} )
\big)\, X^{i}\, Y^{j} \, Z^{k}\,,
\end{align*}
respectively.
\end{thm}

\begin{proof}
It is easy to prove it in coordinates.
\end{proof}

\begin{rmk}\label{Rm: 3.4}
{\rm
Any linear combination of the above operators is an $\mathbb{R}$-bilinear operator, for instance
\begin{align*}
\big(6\, & d(\Alt \psi)   \circ_{3} \varphi
-  6 \, d(\Alt(\psi  \circ_{1}\varphi))
- \Phi_{1}(\varphi,\psi)\big)(X,  Y,Z) =
\\
& =
\big(
\varphi^m_i \,( \partial_j\psi_{mk}
- \partial_k\psi_{mj}
- \partial_m\psi_{jk}
)
+ \varphi^m_k \,( \partial_i\psi_{jm}
- \partial_j\psi_{im}
+ \partial_m\psi_{ij}
)
\\
&  \quad + 
(\psi_{jm} + \psi_{mj}) \, (\partial_{i}\varphi^{m}_{k} - \partial_{k}\varphi^{m}_{i} )
\big) \, X^{i} \, Y^{j} \, Z^{k}\,,
\end{align*}
\begin{align*}
\big(6\, & d(\Alt \psi) \circ_{1} \varphi
-  6 \, d(\Alt(\psi  \circ_{2}\varphi))
- \Phi_{2}(\varphi,\psi)\big)(X,  Y,Z) =
\\
& =
\big(
\varphi^m_i \,( \partial_k\psi_{mj}
- \partial_j\psi_{mk}
- \partial_m\psi_{kj}
)
+ \varphi^m_j \,( \partial_i\psi_{km}
- \partial_k\psi_{im}
+ \partial_m\psi_{ik}
)
\\
&  \quad + 
(\psi_{km} + \psi_{mk}) \, (\partial_{i}\varphi^{m}_{j} - \partial_{j}\varphi^{m}_{i} )
\big) \, X^{i} \, Y^{j} \, Z^{k}\,,
\end{align*}
\begin{align*}
\big(- 6\, & d(\Alt \psi)  \circ_{2} \varphi
+ 6\, d(\Alt \psi)  \circ_{3} \varphi
\\
& \quad
- \Phi_{1}(\varphi,\psi)
+ \Phi_{2}(\varphi,\psi)\big)(X,  Y,Z) =
\\
& =
\big(
\varphi^m_j \,( \partial_i\psi_{km}
- \partial_k\psi_{im}
- \partial_m\psi_{ki}
)
+ \varphi^m_k \,( \partial_j\psi_{mi}
- \partial_i\psi_{mj}
+ \partial_m\psi_{ij}
)
\\
&  \quad + 
(\psi_{im} + \psi_{mi}) \, (\partial_{j}\varphi^{m}_{k} - \partial_{k}\varphi^{m}_{j} )
\big) \, X^{i} \, Y^{j} \, Z^{k}\,,
\end{align*}
and
\begin{align*}
\big(\Phi_{1}(\varphi,\psi) &  + \Phi_{2}(\varphi,\psi)\big)(X, Y,Z) =
\\
& =
\big(
2\,\varphi^{m}_{i}\,  \partial_{m}\psi_{jk} 
+ \varphi^{m}_{j}\,(\partial_{k}\psi_{mi}  - \partial_{i}\psi_{mk} - \partial_{m}\psi_{ik})
\\
&\quad
+ \varphi^{m}_{k}\, (\partial_{j}\psi_{im} - \partial_{i}\psi_{jm} - \partial_{m}\psi_{ji} )
\\
&  \quad
+ (\psi_{im} - \psi_{mi} ) \, ( \partial_{j}\varphi^{m}_{k} - \partial_{k}\varphi^{m}_{j} )
- 2\, \psi_{jm} \, (\partial_{i}\varphi^{m}_{k} - \partial_{k}\varphi^{m}_{i} )
\\
&\quad
- 2\,\psi_{mk} \, ( \partial_{i}\varphi^{m}_{j} -  \partial_{j}\varphi^{m}_{i} )
\big)\, X^{i}\, Y^{j} \, Z^{k}\,
\end{align*}
are such operators which we shall need later.
}
\end{rmk}

\begin{thm}\label{Th: 3.5}
All natural $\mathbb{R}$-bilinear differential operators $\Phi$ transforming a $(1,1)$-tensor field $\varphi$ and a $(0,2)$-tensor field $\psi$ into  $(0,3)$-tensor fields form a 14-parameter family which is a linear combination of operators 
described in Lemma \ref{Lm: 3.3}, Lemma \ref{Lm: 3.2} and Theorem \ref{Th: 3.4}.
\end{thm}

\begin{proof}
By Theorem \ref{Th: 1.2} and \eqref{Eq: 1.3} -- \eqref{Eq: 1.5}
we get that all natural $\mathbb{R}$-bilinear differential operators are of the form
\begin{align*}
\Phi(\varphi,\psi)
& = 
\Phi_{ijk} \, d^{i} \otimes d^{j} \otimes d^{k}\,,  
\end{align*}
where 
\begin{align*}
\Phi_{ijk}
& =
a_1\, \varphi^m_m \, \partial_i\psi_{jk}
+ a_2\, \varphi^m_m \, \partial_i\psi_{kj}
+ a_3\, \varphi^m_m \, \partial_j\psi_{ik}
+ a_4\, \varphi^m_m \, \partial_j\psi_{ki}
\\
& \qquad
+ a_5\, \varphi^m_m \, \partial_k\psi_{ij}
+ a_6\, \varphi^m_m \, \partial_k\psi_{ji}
\\
&\quad
+ a_7\, \varphi^m_i \, \partial_m\psi_{jk}
+ a_8\, \varphi^m_i \, \partial_m\psi_{kj}
+ a_9\, \varphi^m_i \, \partial_j\psi_{mk}
+ a_{10}\, \varphi^m_i \, \partial_j\psi_{km}
\\
& \qquad
+ a_{11}\, \varphi^m_i \, \partial_k\psi_{mj}
+ a_{12}\, \varphi^m_i \, \partial_k\psi_{jm}
\\
& \quad
+ a_{13}\, \varphi^m_j \, \partial_i\psi_{mk}
+ a_{14}\, \varphi^m_j \, \partial_i\psi_{km}
+ a_{15}\, \varphi^m_j \, \partial_m\psi_{ik}
+ a_{16}\, \varphi^m_j \, \partial_m\psi_{ki}
\\
& \qquad
+ a_{17}\, \varphi^m_j \, \partial_k\psi_{im}
+ a_{18}\, \varphi^m_j \, \partial_k\psi_{mi}
\\
& \quad
+ a_{19}\, \varphi^m_k \, \partial_i\psi_{jm}
+ a_{20}\, \varphi^m_k \, \partial_i\psi_{mj}
+ a_{21}\, \varphi^m_k \, \partial_j\psi_{im}
+ a_{22}\, \varphi^m_k \, \partial_j\psi_{mi}
\\
& \qquad
+ a_{23}\, \varphi^m_k \, \partial_m\psi_{ij}
+ a_{24}\, \varphi^m_k \, \partial_m\psi_{ji}
\\
& \quad
+ b_{1}\, \psi_{ij} \, \partial_k\varphi^m_m 
+ b_{2}\, \psi_{ji} \, \partial_k\varphi^m_m 
+ b_{3}\, \psi_{ik} \, \partial_j\varphi^m_m 
+ b_{4}\, \psi_{ki} \, \partial_j\varphi^m_m 
\\
& \qquad
+ b_{5}\, \psi_{jk} \, \partial_i\varphi^m_m 
+ b_{6}\, \psi_{kj} \, \partial_i\varphi^m_m 
\\
& \quad
+ b_{7}\, \psi_{mj} \, \partial_k\varphi^m_i 
+ b_{8}\, \psi_{jm} \, \partial_k\varphi^m_i 
+ b_{9}\, \psi_{mk} \, \partial_j\varphi^m_i 
+ b_{10}\, \psi_{km} \, \partial_j\varphi^m_i 
\\
& \qquad
+ b_{11}\, \psi_{jk} \, \partial_m\varphi^m_i 
+ b_{12}\, \psi_{kj} \, \partial_m\varphi^m_i
\\
& \quad
+ b_{13}\, \psi_{im} \, \partial_k\varphi^m_j 
+ b_{14}\, \psi_{mi} \, \partial_k\varphi^m_j
+ b_{15}\, \psi_{ik} \, \partial_m\varphi^m_j 
+ b_{16}\, \psi_{ki} \, \partial_m\varphi^m_j
\\
& \qquad
+ b_{17}\, \psi_{mk} \, \partial_i\varphi^m_j 
+ b_{18}\, \psi_{km} \, \partial_i\varphi^m_j
\\
& \quad
+ b_{19}\, \psi_{ij} \, \partial_m\varphi^m_k 
+ b_{20}\, \psi_{ji} \, \partial_m\varphi^m_k
+ b_{21}\, \psi_{im} \, \partial_j\varphi^m_k 
+ b_{22}\, \psi_{mi} \, \partial_j\varphi^m_k
\\
& \qquad
+ b_{23}\, \psi_{jm} \, \partial_i\varphi^m_k 
+ b_{24}\, \psi_{mj} \, \partial_i\varphi^m_k
\,.
\end{align*}
In order to calculate relations for coefficients $ a_i,\,b_i $, $i=1,\dots,24$, we use the method of an auxiliary linear symmetric connection $K$, \cite[p. 144]{KruJan90}.
We replace derivatives of tensor fields with covariant derivatives and assume that the operator is independent of $K$. Then we get
\begin{align*}
0
& =
\varphi^m_m \, \big[ (a_1+a_3)\, K_i{}^p{}_j \, \psi_{pk} 
+ (a_2+a_5)\, K_i{}^p{}_k \, \psi_{pj}
+ (a_2+a_4)\, K_i{}^p{}_j \, \psi_{kp}
\\
&\qquad
+ (a_1+a_6)\, K_i{}^p{}_k \, \psi_{jp}
+ (a_3+a_5)\, K_j{}^p{}_k \, \psi_{ip}
+ (a_4+a_6)\, K_j{}^p{}_k \, \psi_{pi}
\big]
\\
& \quad
+ \varphi^m_i \, \big[ (a_7+a_9 - b_9)\, K_j{}^p{}_m \, \psi_{pk} 
+ (a_8+a_{10} - b_{10})\, K_j{}^p{}_m \, \psi_{kp}
\\
&\qquad
+ (a_7+a_{12} - b_8)\, K_m{}^p{}_k \, \psi_{jp}
+ (a_8+a_{11} - b_7)\, K_m{}^p{}_k \, \psi_{pj}
\\
&\qquad
+ (a_9+a_{11})\, K_j{}^p{}_k \, \psi_{mp}
+ (a_{10}+a_{12})\, K_j{}^p{}_k \, \psi_{pm}
\\
& \qquad
- b_{12} \, K_p{}^p{}_m \, \psi_{kj}
- b_{11} \, K_p{}^p{}_m \, \psi_{jk}
\big]
\\
& \quad
+ \varphi^m_j \, \big[ (a_{13} + a_{15} - b_{17})\, K_i{}^p{}_m \, \psi_{pk} 
+ (a_{14} + a_{16} - b_{18})\, K_i{}^p{}_m \, \psi_{kp}
\\
&\qquad
+ (a_{15} + a_{17} - b_{13})\, K_m{}^p{}_k \, \psi_{ip}
+ (a_{16} + a_{18} - b_{14})\, K_m{}^p{}_k \, \psi_{pi}
\\
&\qquad
+ (a_{13} + a_{18})\, K_i{}^p{}_k \, \psi_{mp}
+ (a_{14}+a_{17})\, K_i{}^p{}_k \, \psi_{pm}
\\
& \qquad
- b_{16} \, K_p{}^p{}_m \, \psi_{ki}
- b_{15} \, K_p{}^p{}_m \, \psi_{ik} \big]
\end{align*}
\begin{align*}
& \quad
+ \varphi^m_k \, \big[ (a_{20} + a_{23} - b_{24})\, K_i{}^p{}_m \, \psi_{pj} 
+ (a_{19} + a_{24} - b_{23})\, K_i{}^p{}_m \, \psi_{jp}
\\
&\qquad
+ (a_{21} + a_{23} - b_{21})\, K_m{}^p{}_j \, \psi_{ip}
+ (a_{22} + a_{24} - b_{22})\, K_m{}^p{}_j \, \psi_{pi}
\\
&\qquad
+ (a_{20} + a_{22})\, K_i{}^p{}_j \, \psi_{mp}
+ (a_{19}+a_{21})\, K_i{}^p{}_j \, \psi_{pm}
\\
& \qquad
- b_{20} \, K_p{}^p{}_m \, \psi_{ji}
- b_{19} \, K_p{}^p{}_m \, \psi_{ij} \big]
\\
& \quad
+ \varphi^m_p \, \big[b_{19} \, K_m{}^p{}_k \,\psi_{ij} 
+ b_{20}\,  K_m{}^p{}_k \,\psi_{ji}
+ b_{15}\,  K_m{}^p{}_j \,\psi_{ik}
+ b_{16}\,  K_m{}^p{}_j \,\psi_{ki}
\\
& \qquad
+ b_{11}\,  K_m{}^p{}_i \,\psi_{jk}
+ b_{12}\,  K_m{}^p{}_i \,\psi_{kj}
\\
& \quad
+ (b_{13} + b_{21} ) \,  K_j{}^p{}_k \,\psi_{im}
+ (b_{14} + b_{22} ) \,  K_j{}^p{}_k \,\psi_{mi}
+ (b_{8} + b_{23} ) \,  K_i{}^p{}_k \,\psi_{jm}
\\
& \qquad
+ (b_{7} + b_{24} ) \,  K_i{}^p{}_k \,\psi_{mi}
+ (b_{10} + b_{18} ) \,  K_i{}^p{}_j \,\psi_{km}
+ (b_{9} + b_{17} ) \,  K_i{}^p{}_j \,\psi_{mk}
\big]\,.
\end{align*}
So, the operator is independent of $K$ if and only if the following conditions for coefficients are satisfied:

\medskip
I: Coefficients $b_{1}, \dots , b_{6}$ are arbitrary and we obtain that the corresponding  part of the operator
$\Phi({\varphi},\psi) $
is a linear combination of operators from Lemma \ref{Lm: 3.2}. We shall put $ B_i = b_{i}\,, $
$ i = 1,\dots, 6 $.

\medskip
II: For part staying with $\varphi^{m}_{m}$ we get that the coefficients $a_{1}, \dots , a_{6}$ satisfy the conditions
$$
\begin{matrix}
a_{1} & & + a_{3} & & & & = & 0\,,
\\
a_{1} & &  & & & + a_{6} & = & 0 \,,
\\
 & a_{2} &  & & + a_{5} & & = & 0 \,,
\\
 & a_{2} &  & + a_{4} &  & & = & 0 \,,
\\
 &  & a_{3}  &  & + a_{5} & & = & 0 \,,
\\
 & &  &  a_{4} &  & + a_{6} & = & 0   \,.
\end{matrix}
$$
This system of equations has one free variable and putting $a_{6} = B_{7}$
and the others free variables are vanishing we get
 a multiple of the operator
$$
\Phi_{ijk} = \varphi^m_m \,\big( \partial_i\psi_{jk}
- \partial_i\psi_{kj}
+ \partial_j\psi_{ki}
- \partial_j\psi_{ik}
+ \partial_k\psi_{ij}
- \partial_k\psi_{ji}
\big)
$$
which is a multiple of the  operator 
$$
 (\tr\varphi)\, d(\Alt \psi)\,
$$
from Lemma \ref{Lm: 3.3}.

\medskip
III: For part staying with $\varphi^{p}_{m}$ we get the following conditions.
The coefficients $b_{11} = b_{12} = b_{15} = b_{16} = b_{19} = b_{20} = 0$.

Further 
\begin{align*}
b_{24} & = - b_{7}\,, \qquad   b_{23}  = - b_{8}\,, \qquad  b_{22}  = - b_{14}\,,
\\
b_{21} & = - b_{13}\,,  \qquad b_{18}  = - b_{10}\,,  \qquad b_{17}  = - b_{9}\,.
\end{align*}

\medskip
IV: For part staying with $\varphi^{p}_{i}$
the coefficients $a_{7}, \dots , a_{12}$ satisfy
$$
\begin{matrix}
a_{7} & & + a_{9} & & & & = & b_{9} \,,
\\
a_{7} & &  & & & + a_{12}& = & b_{8} \,,
\\
 & a_{8} &  & & + a_{11}&& = & b_{7} \,,
\\
 & a_{8} &  & + a_{10}&  & & = & b_{10} \,,
\\
 &  & a_{9}  &  & + a_{11}& & = & 0  \,,
\\
 & &  &  a_{10}&  & + a_{12}& = & 0    \,.
\end{matrix}
$$

\medskip
V: For part staying with $\varphi^{p}_{j}$
the coefficients $a_{13}, \dots , a_{18}$ satisfy
$$
\begin{matrix}
a_{13} & & + a_{15} & & & & = & b_{17} \,,
\\
a_{13} & &  & & & + a_{18}& = & 0 \,,
\\
 & a_{14} &  & & + a_{17}&& = & 0 \,,
\\
 & a_{14} &  & + a_{16}&  & & = & b_{18} \,,
\\
 &  & a_{15}  &  & + a_{17}& & = & b_{13} \,,
\\
 & &  &  a_{16}&  & + a_{18}& = & b_{14}    \,.
\end{matrix}
$$

\medskip
VI: For part staying with $\varphi^{p}_{k}$
the coefficients $a_{19}, \dots , a_{24}$ satisfy
$$
\begin{matrix}
a_{19} & & + a_{21} & & & & = & 0 \,,
\\
a_{19} & &  & & & + a_{24}& = & b_{23}  \,,
\\
 & a_{20} &  & & + a_{23}&& = & b_{24}  \,,
\\
 & a_{20} &  & + a_{22}&  & & = & 0  \,,
\\
 &  & a_{21}  &  & + a_{23}& & = & b_{21}  \,, 
\\
 & &  &  a_{22}&  & + a_{24}& = & b_{22}    \,.
\end{matrix}
$$

The above systems IV -- VI of linear equations  we modify to
$$
\begin{matrix}
a_{7} & & + a_{9} & & & & = & b_{9} \,,
\\
 & a_{8} &  & + a_{10}&  & & = & b_{10} \,,
\\
 &  & a_{9}  &  & + a_{11}& & = & 0  \,,
 \\
 & &  &  a_{10}&  & + a_{12}& = & 0   \,,
\\
 & &  & &  a_{11}& + a_{12} & = & b_{7} - b_{10}  \,,
 \\
 & &  & & & 0 & = & b_{7} - b_{10} - b_{8} + b_{9} \,.
\end{matrix}
$$

$$
\begin{matrix}
a_{13} & & + a_{15} & & & & = & b_{17} \,,
\\
 & a_{14} &  & + a_{16}&  & & = & b_{18} \,,
\\
 &  & a_{15}  &  & + a_{17}& & = & b_{13} \,,
 \\
 & &  &  a_{16}&  & + a_{18}& = & b_{14}  \,,
\\
 & &  & &  a_{17}& + a_{18} & = & b_{13} - b_{17}  \,,
 \\
 & &  & & & 0 & = & b_{14} - b_{18} - b_{13} + b_{17}  \,.
\end{matrix}
$$

$$
\begin{matrix}
a_{19} & & + a_{21} & & & & = & 0 \,,
\\
 & a_{20} &  & + a_{22}&  & & = & 0 \,,
\\
 &  & a_{21}  &  & + a_{23}& & = & b_{21} \,, 
\\
 & &  &  a_{22}&  & + a_{24}& = & b_{22}    \,,
\\
 & &  & &  a_{23}& + a_{24} & = & b_{21} + b_{23}  \,,
 \\
 & &  & & & 0 & = & b_{22} + b_{24} - b_{21} - b_{23} \,,
\end{matrix}
$$
So, for coefficients $b_{7}\,,\,\,b_{8}\,,\,\,b_{9}\,,\,\,b_{10}\,,\,\,b_{13}\,,\,\,b_{14}\,,\,\,b_{17}\,,\,\,b_{18}\,,\,\,b_{21}\,,\,\,b_{22}\,,$ $b_{23}\,,$ $b_{24}$ we have a system of homogeneous linear equations with 4 independent variables. We choose as these free variables $ b_{18} = B_{8} $,  $ b_{22} = B_{9} $,  $ b_{23} = B_{10} $ and  $ b_{24} = B_{11} $. We obtain
\begin{align*}
b_{7} 
& =
- B_{11} \,,
&
b_{18} 
& =
B_{8}\,,
\\
b_{8} 
& =
- B_{10}\,,
&
b_{17} 
& =
B_{8} - B_{10} + B_{11}\,,
\\
b_{9} 
& =
- B_{8} - B_{10} + B_{11}\,,
&
b_{21} 
& =
B_{9} - B_{10} + B_{11}\,,
\\
b_{10} 
& =
- B_{8}\,,
&
b_{22} 
& =
B_{9}\,,
\\
b_{13} 
& =
- B_{9} + B_{10} - B_{11}\,,
&
b_{23} 
& =
B_{10}\,,
\\
b_{14} 
& =
- B_{9}\,,
&
b_{24} 
& =
B_{11}\,.
\end{align*}

Now, putting $a_{12}$ as a free variable $ B_{12} $, we get from the system of equations IV
\begin{align*}
a_{7} 
& =
- B_{12} - B_{10}\,,
& 
a_{10} 
& =
- B_{12}\,,
\\
a_{8} 
& =
B_{12} - B_{8}\,,
&
a_{11} 
& =
- B_{12} + B_{8} - B_{11}\,,
\\
a_{9} 
& =
B_{12} - B_{8} + B_{11}\,,
&
a_{12} 
& =
B_{12}\,.
\end{align*}

Further, putting $a_{18}$ as a free variable $ B_{13} $, we get from the system of equations V
\begin{align*}
a_{13} 
& =
- B_{13}\,,
&
a_{16} 
& =
- B_{13} - B_{9}\,,
\\
a_{14} 
& =
B_{13} + B_{8} + B_{9}\,,
&
a_{17} 
& =
- B_{13} - B_{8} - B_{9} \,,
\\
a_{15} 
& =
B_{13} + B_{8} + B_{10} - B_{11}\,,
&
a_{18} 
& =
B_{13}\,.
\end{align*}

Finally, putting $a_{24}$ as a free variable $ B_{14} $, we get from the system of equations VI
\begin{align*}
a_{19} 
& =
- B_{14} + B_{10}\,,
&
a_{22} 
& =
- B_{14} + B_{9}\,,
\\
a_{20} 
& =
B_{14} - B_{9}\,,
&
a_{23} 
& =
- B_{14} + B_{9} + B_{11}\,,
\\
a_{21} 
& =
B_{14} - B_{10} \,,
&
a_{24} 
& =
B_{14}\,.
\end{align*}

Let us put $B_{12} = 1$ and the others free variables are vanishing. We get
$$
\Phi_{ijk} = - \varphi^m_i \,\big( \partial_m\psi_{jk}
- \partial_m\psi_{kj}
+  \partial_j\psi_{km}
- \partial_j\psi_{mk}
+  \partial_k\psi_{mj}
- \partial_k\psi_{jm}
\big)  
$$
which is the coordinate expression for a multiple of
$$
d(\Alt\psi) \circ_1\varphi\,.
$$
Similarly for  $B_{13} =1$ (respective  $B_{14} =1$) and the others free variables vanishing we get multiples of
$d(\Alt\psi) \circ_2\varphi$ (respective $d(\Alt\psi) \circ_3\varphi$).

\smallskip

If we put $ B_{8} = 1 $ and the others free variables vanishing we get
\begin{align*}
\Phi_{ijk}
& =
\varphi^m_i \, (\partial_k\psi_{mj} - \partial_m\psi_{kj} - \partial_j\psi_{mk})
+ \varphi^m_j \, ( \partial_i\psi_{km} + \partial_m\psi_{ik} - \partial_k\psi_{im})
\\
& 
+ (\psi_{mk} + \psi_{km})\,(\partial_i\varphi^m_j - \partial_j\varphi^m_i)
 \,.
\end{align*}
According to Remark \ref{Rm: 3.4} this operator corresponds in coordinates to 
$$
6 \, d(\Alt \psi)\circ_{1}\varphi -
6 \, d(\Alt(\psi\circ_{2}\varphi))
- \Phi_{2}(\varphi,\psi)\,. 
$$

\smallskip

If we put $ B_{9} = 1 $ and the others free variables vanishing we get
\begin{align*}
\Phi_{ijk}
& =
\varphi^m_j \, (\partial_i\psi_{km} - \partial_m\psi_{ki} - \partial_k\psi_{im})
+ \varphi^m_k \, ( - \partial_i\psi_{mj} + \partial_j\psi_{mi} + \partial_m\psi_{ij})
\\
&
+ (\psi_{im}+ \psi_{mi} ) \, (\partial_j\varphi^m_k  - \partial_k\varphi^m_j) 
\,.
\end{align*}
According to Remark \ref{Rm: 3.4} this operator corresponds in coordinates to 
$$
- 6\, d(\Alt \psi)\circ_{2}\varphi
+ 6 \, d(\Alt \psi)\circ_{3}\varphi
- \Phi_{1}(\varphi,\psi)
+ \Phi_{2}(\varphi,\psi)\,. 
$$

\smallskip

If we put $ B_{10} = 1 $ and the others free variables vanishing we get
\begin{align*}
\Phi_{ijk}
=
& - \varphi^m_i \, \partial_m\psi_{jk}
+ \varphi^m_j \,  \partial_m\psi_{ik}
+ \varphi^m_k \, ( \partial_i\psi_{jm} - \partial_j\psi_{im})
\\
& 
+ \psi_{jm} \, ( \partial_i\varphi^m_k  - \partial_k\varphi^m_i)
- \psi_{mk} \, (\partial_j\varphi^m_i -  \partial_i\varphi^m_j )
\\
& 
+ \psi_{im} \, ( \partial_k\varphi^m_j - \partial_j\varphi^m_k) 
 \,. 
\end{align*}

If we put $ B_{11} = 1 $ and the others free variables vanishing we get
\begin{align*}
\Phi_{ijk}
& =
 \varphi^m_i \, (\partial_j\psi_{mk} - \partial_k\psi_{mj})
- \varphi^m_j \, \partial_m\psi_{ik} 
+ \varphi^m_k \, \partial_m\psi_{ij}
\\
& \quad
+ \psi_{mj} \, ( \partial_i\varphi^m_k - \partial_k\varphi^m_i )
+ \psi_{mk} \, ( \partial_j\varphi^m_i - \partial_i\varphi^m_j ) 
\\
& \quad
+ \psi_{im} \, ( \partial_j\varphi^m_k - \partial_k\varphi^m_j ) \,.
\end{align*}

Then the sum of the last 2 operators, i.e. $ B_{10} = 1 = B_{11} $, gives
\begin{align*}
\Phi_{ijk}
 & =
 \varphi^m_i \, (\partial_j\psi_{mk} - \partial_k\psi_{mj}
     - \partial_m\psi_{jk})
+ \varphi^m_k \, ( \partial_m\psi_{ij} + \partial_i\psi_{jm} 
    - \partial_j\psi_{im})
\\
& \quad
+ (\psi_{jm} + \psi_{mj}) \, ( \partial_i\varphi^m_k 
    - \partial_k\varphi^m_i ) \,.
\end{align*}
According to Remark \ref{Rm: 3.4} this operator corresponds in coordinates to 
$$  
6\, d(\Alt \psi)\circ_{3} \varphi
- 6\, d(\Alt(\psi  \circ_{1}\varphi))
- \Phi_{1}(\varphi,\psi)   \,.
$$

On the other side for $ B_{10} = 1\,,\,\,  B_{11} = - 1 $ we have
\begin{align*}
\Phi_{ijk}
& = 
 \varphi^m_i \, ( - \partial_j\psi_{mk} + \partial_k\psi_{mj}
     - \partial_m\psi_{jk})
 + 2 \, \varphi^m_j \, \partial_m\psi_{ik}
 \\
 & \quad
+ \varphi^m_k \, ( - \partial_m\psi_{ij} + \partial_i\psi_{jm} 
    - \partial_j\psi_{im})
+ 2\, \psi_{im} \, ( \partial_k\varphi^m_j - \partial_j\varphi^m_k)
\\
& \quad 
+ (\psi_{jm} - \psi_{mj}) \, ( \partial_i\varphi^m_k 
    - \partial_k\varphi^m_i )
- 2 \, \psi_{mk} \, (\partial_j\varphi^m_i -  \partial_i\varphi^m_j )
 \,
\end{align*}
which, according to Remark \ref{Rm: 3.4}, corresponds in coordinates to
$$  
\Phi_{1}(\varphi,\psi) + \Phi_{2}(\varphi,\psi)\,.
$$
So all 14 independent operators are generated by 14 independent operators described in Lemma \ref{Lm: 3.3}, Lemma \ref{Lm: 3.2} and Theorem \ref{Th: 3.4}.
\end{proof}

\begin{rmk}\label{Rm: 3.5}
{\rm
Let $\psi$ be a 2--form. According to \cite[p. 69]{KolMicSlo93} we can define the Lie derivative of  $ \psi $ with respect to $ \varphi $ as
\begin{equation*}
L_{\varphi}\psi = [i_{\varphi},d]\psi = i_{\varphi} d\psi - d i_{\varphi}\psi\,
\end{equation*}
which is a 3-form. It is easy to see that
\begin{equation*}
L_{\varphi}\psi = d\psi \circ_{1} \varphi + 
d\psi \circ_{2} \varphi + d\psi \circ_{3} \varphi
- 2 \, d(\Alt(\psi \circ_1\varphi)) \,.
\end{equation*}
}
\end{rmk}

\section{Natural operators transforming (1,2)-tensor fields $ S $ and 1-forms $ \psi $ into (0,3)-tensor fields}

Let us recall that, according to Theorem \ref{Th: 1.1}, all natural operators transforming (1,2)-tensor fields $ S $ and 1-forms $ \psi $ into (0,3)-tensor fields are $\mathbb{R}$-bilinear and of order 1.

\subsection{General case}

We shall denote by $C^{1}_{i}S$, $ i = 1,2 $, the contraction with respect to the
corresponding indices.

\begin{lem}\label{Lm: 4.1}
We have 6 canonical natural differential operators given by 
$ C^{1}_{i}S\otimes d\psi $, $ i = 1,2 $, namely
\begin{align*}
&
(C^1_1  S)(X) \, d\psi(Y,Z)\,,\quad 
(C^1_1   S)(Y) \, d\psi(X,Z)\,,\quad
(C^1_1   S)(Z) \, d\psi(X,Y)\,,\quad
\\
&
(C^1_2 S)(X) \, d\psi(Y,Z)\,,\quad 
(C^1_2 S)(Y) \, d\psi(X,Z)\,,\quad
(C^1_2 S)(Z) \, d\psi(X,Y)\,.
\end{align*}
\end{lem}

\begin{lem}\label{Lm: 4.2}
We have 6 canonical natural differential operators given by the composition of $ S $ with $ d\psi $, namely
\begin{align*}
&
 d\psi (S(X,Y),Z)\,,\quad   d\psi (S(Y,X),Z)\,,\quad  d\psi (S(X,Z),Y)\,,
 \\
 &
d\psi (S(Z,X),Y)\,,\quad   d\psi (S(Y,Z),X)\,,\quad  d\psi (S(Z,Y),X)\,.
\end{align*}
\end{lem}

\begin{lem}\label{Lm: 4.3}
We have 6 canonical natural differential operators given by $ \psi \otimes d(C^{1}_{i}S) $, $ i= 1,2 $, namely
\begin{align*}
&
 \psi (X) d(C^1_1  S)(Y,Z)\,,\quad    \psi (Y) d(C^1_1  S)(X,Z)\,,\quad \psi (Z) d(C^1_1  S)(X,Y)\,,
 \\
 &
 \psi (X) d(C^{1}_{2}S)(Y,Z)\,,\quad    \psi (Y) d(C^{1}_{2}S)(X,Z)\,,\quad \psi (Z) d(C^{1}_{2}S)(X,Y\,.
\end{align*}
\end{lem}

\begin{lem}\label{Lm: 4.4}
Let us assume the antisymmetric part $ \Alt S $ of $ S $ with the coordinate expression
$ \Alt S = \tfrac 12 (S^{i}_{jk} - S^{i}_{kj}) \, \partial_{i} \otimes d^{j} \otimes d^{k} $.
Then
\begin{equation}
d(\psi \circ \Alt S) 
\end{equation}
is a first order natural $\mathbb{R}$-bilinear differential operator with values in 3-forms.
\end{lem}

\begin{crl}
If the 1--form $ \psi $ is closed then we have only 7 operators from Lemma \ref{Lm: 4.3} and Lemma \ref{Lm: 4.4}.
\end{crl}

\begin{thm}\label{Th: 4.1}
All natural differential operators transforming a (1,2)-tensor field $ S $ and a 1-form $ \psi $ into (0,3)-tensor fields form a 19-parameter family of operators
described in Lemmas \ref{Lm: 4.1} -- \ref{Lm: 4.4}. 
\end{thm}

\begin{proof}
According to Theorem \ref{Th: 1.1} and \eqref{Eq: 1.3} -- \eqref{Eq: 1.5} 
\begin{align*}
\Phi(S,\psi)
& = 
\Phi_{ijk}  \, d^{i} \otimes d^{j} \otimes d^{k}\,,  
\end{align*}
where 
\begin{align*}
\Phi_{ijk}
&=
a_1\, S^{m}_{mi} \, \partial_j\psi_{k}
+ a_2\, S^{m}_{mi} \, \partial_k\psi_{j}
+ a_3\, S^{m}_{mj} \, \partial_i\psi_{k}
+ a_4\, S^{m}_{mj} \, \partial_k\psi_{i}
\\
& \qquad
+ a_5\, S^{m}_{mk} \, \partial_i\psi_{j}
+ a_6\, S^{m}_{mk} \, \partial_j\psi_{i}
\\
&\quad
+ a_7\, S^{m}_{im} \, \partial_j\psi_{k}
+ a_8\, S^{m}_{im} \, \partial_k\psi_{j}
+ a_9\, S^{m}_{jm} \, \partial_i\psi_{k}
+ a_{10}\, S^{m}_{jm} \, \partial_k\psi_{i}
\\
& \qquad
+ a_{11}\, S^{m}_{km} \, \partial_i\psi_{j}
+ a_{12}\, S^{m}_{km} \, \partial_j\psi_{i}
\\
& \quad
+ a_{13}\, S^{m}_{ij} \, \partial_m\psi_{k}
+ a_{14}\, S^{m}_{ji} \, \partial_m\psi_{k}
+ a_{15}\, S^{m}_{ik} \, \partial_m\psi_{j}
+ a_{16}\, S^{m}_{ki} \, \partial_m\psi_{j}
\\
& \qquad
+ a_{17}\, S^{m}_{jk} \, \partial_m\psi_{i}
+ a_{18}\, S^{m}_{kj} \, \partial_m\psi_{i}
\\
& \quad
+ a_{19}\, S^{m}_{ij} \, \partial_k\psi_{m}
+ a_{20}\, S^{m}_{ji} \, \partial_k\psi_{m}
+ a_{21}\, S^{m}_{ik} \, \partial_j\psi_{m}
+ a_{22}\, S^{m}_{ki} \, \partial_j\psi_{m}
\\
& \qquad
+ a_{23}\, S^{m}_{jk} \, \partial_i\psi_{m}
+ a_{24}\, S^{m}_{kj} \, \partial_i\psi_{m}
\\
& \quad
+ b_{1}\, \psi_{i} \, \partial_j S^{m}_{km}
+ b_{2}\, \psi_{i} \, \partial_j S^{m}_{mk}
+ b_{3}\, \psi_{i} \, \partial_k S^{m}_{jm}
+ b_{4}\, \psi_{i} \, \partial_k S^{m}_{mj} 
\\
& \qquad
+ b_{5}\, \psi_{i} \, \partial_m S^{m}_{jk}
+ b_{6}\, \psi_{i} \, \partial_m S^{m}_{kj}
\\
& \quad
+ b_{7}\, \psi_{j} \, \partial_i S^{m}_{km}
+ b_{8}\, \psi_{j} \, \partial_i S^{m}_{mk}
+ b_{9}\, \psi_{j} \, \partial_k S^{m}_{im} 
+ b_{10}\, \psi_{j} \, \partial_k S^{m}_{mi} 
\\
& \qquad
+ b_{11}\, \psi_{j} \, \partial_m S^{m}_{ik}
+ b_{12}\, \psi_{j} \, \partial_m S^{m}_{ki}
\\
& \quad
+ b_{13}\, \psi_{k} \, \partial_i S^{m}_{jm} 
+ b_{14}\, \psi_{k} \, \partial_i S^{m}_{mj} 
+ b_{15}\, \psi_{k} \, \partial_j S^{m}_{im} 
+ b_{16}\, \psi_{k} \, \partial_j S^{m}_{mi} 
\\
& \qquad
+ b_{17}\, \psi_{k} \, \partial_m S^{m}_{ij}  
+ b_{18}\, \psi_{k} \, \partial_m S^{m}_{ji} 
\\
& \quad
+ b_{19}\, \psi_{m} \, \partial_i S^{m}_{jk} 
+ b_{20}\, \psi_{m} \, \partial_i S^{m}_{kj}
+ b_{21}\, \psi_{m} \, \partial_j S^{m}_{ik} 
+ b_{22}\, \psi_{m} \, \partial_j S^{m}_{ki}
\\
& \qquad
+ b_{23}\, \psi_{m} \, \partial_k S^{m}_{ij}
+ b_{24}\, \psi_{m} \, \partial_k S^{m}_{ji}\,.
\end{align*}
In order to calculate relations for coefficients $ a_i,\,b_i $, $i=1,\dots,24$, we use the method of an auxiliary linear symmetric connection $K$, \cite{KruJan90}.
We replace derivatives of tensor fields with covariant derivatives and assume that the operator is independent of $K$. Then we get
\begin{align*}
0
& =
\psi_{p} \, \big[ S^{m}_{mi} \, \big( a_1 
+  
a_2 \big) \, K_k{}^p{}_j 
+ 
S^{m}_{im} \, \big( a_7 
+  
a_8 \big)\, K_k{}^p{}_j  
\\
& \quad +
S^{m}_{mj} \, \big( a_3 
+  
a_4 \big) \, K_k{}^p{}_i  
+ 
S^{m}_{jm} \, \big( a_9 
+  
a_{10} \big) \, K_k{}^p{}_i 
\\
& \quad +
S^{m}_{mk} \, \big( a_5 
+  
a_6 \big) \, K_j{}^p{}_i  
+ 
S^{m}_{km} \, \big( a_{11} 
+  
a_{12} \big) \, K_j{}^p{}_i 
\\
& \quad +
S^{m}_{ij} \, \big( a_{13} 
+  
a_{19} - b_{23}\big) \, K_k{}^p{}_m  
+ 
S^{m}_{ji} \, \big( a_{14} 
+  
a_{20} - b_{24}\big) \, K_k{}^p{}_m 
\end{align*}
\begin{align*}
& \quad +
S^{m}_{ik} \, \big( a_{15} 
+  
a_{21}  - b_{21}\big) \, K_j{}^p{}_m  
+ 
S^{m}_{ki} \, \big( a_{16} 
+  
a_{22} - b_{22}\big) \, K_j{}^p{}_m 
\\
& \quad +
S^{m}_{jk} \, \big( a_{17} 
+  
a_{23} - b_{19}\big) \, K_i{}^p{}_m 
+ 
S^{m}_{kj} \, \big( a_{18} 
+  
a_{24} - b_{20}\big) \, K_i{}^p{}_m \big] 
\\
& \quad +
\psi_{m} \, \big[ (b_{19} + b_{21})\,  K_i{}^p{}_j \, S^{m}_{pk} + (b_{19} + b_{24})\,K_i{}^p{}_k \, S^{m}_{jp}    
\\
& \qquad +
    (b_{20} + b_{23}) \,  K_i{}^p{}_k \, S^{m}_{pj} + (b_{20} + b_{22}) \,  K_i{}^p{}_j \, S^{m}_{kp}
\\
& \qquad +
    (b_{21} + b_{23}) \,  K_j{}^p{}_k \, S^{m}_{ip}
 +
    (b_{22} + b_{24}) \,  K_j{}^p{}_k \, S^{m}_{pi} 
\big]    
\\
& \quad +
\psi_{i} \, \big[ (b_{1} + b_{3}) \,  K_j{}^p{}_k \, S^{m}_{pm} + ( b_{2} + b_{4}) \,  K_j{}^p{}_k \, S^{m}_{mp} 
\\
& \qquad +
    b_{5}\, ( K_m{}^p{}_j \, S^{m}_{pk} + K_m{}^p{}_k \, S^{m}_{jp}
    - K_m{}^m{}_p \, S^{p}_{jk})
    \\
& \qquad +
    b_{6}\, ( K_m{}^p{}_k \, S^{m}_{pj} + K_m{}^p{}_j \, S^{m}_{kp}
    - K_m{}^m{}_p \, S^{p}_{kj})
\big]
\\
& \quad +
\psi_{j} \, \big[ ( b_{7} + b_{9} ) \,  K_i{}^p{}_k \, S^{m}_{pm} + ( b_{8}  + b_{10} ) \,  K_i{}^p{}_k \, S^{m}_{mp} 
\\
& \qquad +
    b_{11}\, ( K_m{}^p{}_i \, S^{m}_{pk} + K_m{}^p{}_k \, S^{m}_{ip}
    - K_m{}^m{}_p \, S^{p}_{ik})
    \\
& \qquad +
    b_{12}\, ( K_m{}^p{}_k \, S^{m}_{pi} + K_m{}^p{}_i \, S^{m}_{kp}
    - K_m{}^m{}_p \, S^{p}_{ki})
\big]
\\
& \quad +
\psi_{k} \, \big[ ( b_{13} + b_{15} ) \,  K_i{}^p{}_j \, S^{m}_{pm} + ( b_{14}  + b_{16} ) \,  K_i{}^p{}_j \, S^{m}_{mp} 
\\
& \qquad +
    b_{17}\, ( K_m{}^p{}_i \, S^{m}_{pj} + K_m{}^p{}_j \, S^{m}_{ip}
    - K_m{}^m{}_p \, S^{p}_{ij})
    \\
& \qquad +
    b_{18}\, ( K_m{}^p{}_j \, S^{m}_{pi} + K_m{}^p{}_i \, S^{m}_{jp}
    - K_m{}^m{}_p \, S^{p}_{ji})
\big]\,.
\end{align*}

So, we get $b_{3} = - b_{1}$, $b_{4} = - b_{2}$, $b_{9} = - b_{7}$, $b_{10} = - b_{8}$, $b_{16} = - b_{14}$, $b_{15} = - b_{13}$ and $b_{5} =  b_{6} =  b_{11} =  b_{12} =  b_{17} =  b_{18} = 0$. This corresponds to linear combination of operators from Lemma \ref{Lm: 4.3}.

Further $a_{2} = - a_{1}$, $a_{4} = - a_{3}$, $a_{8} = - a_{7}$, $a_{10} = - a_{9}$,  $a_{6} = - a_{5}$ and  $a_{12} = - a_{11}$ which gives a linear combinations of operators from Lemma \ref{Lm: 4.1}.

For coefficients $ b_{19}, \dots , b_{24} $ we obtain the following system of homogeneous linear equations.
$$
\begin{matrix}
b_{19} & & + b_{21} & & & & = & 0\,,
\\
b_{19} & &  & & & + b_{24} & = & 0\,,
\\
 & b_{20} &  & & + b_{23} & & = & 0\,,
\\
 & b_{20} &  & + b_{22} &  & & = & 0 \,,
\\
 &  & b_{21}  &  & + b_{23} & & = & 0 \,,
\\
 & &  &  b_{22} &  & + b_{24} & = & 0  \,. 
\end{matrix}
$$
This system of equations has one free variable and if we put $ b_{24} = - B $ we obtain
$b_{19} = b_{22} = b_{23} = B$ and $ b_{20} = b_{21} = b_{24} = - B $. 

Finally, we have the system of linear equations
$$
\begin{matrix}
a_{13}  +  a_{19}  -  b_{23}  =  0\,,
\quad &
a_{14}  +  a_{20}  -  b_{24}  =  0\,,
\\
a_{15}  +  a_{21}  -  b_{21}  =  0\,,
\quad &
a_{16}  +  a_{22}  -  b_{22}  =  0 \,,
\\
a_{17}  +  a_{23}  -  b_{19}  =  0\,,
\quad &
a_{18}  +  a_{24}  -  b_{20}  =  0 \,.  
\end{matrix}
$$
It gives the following operators
\begin{align*}
\Phi_{ijk}
&  =
a_{13}\, S^{m}_{ij} \,( \partial_m\psi_{k} - \partial_k\psi_{m})
+ B\, (S^{m}_{ij} \, \partial_k\psi_{m} + \psi_{m} \, \partial_k S^{m}_{ij})
\\
& \quad
+ a_{14}\, S^{m}_{ji} \, ( \partial_m\psi_{k} - \partial_k\psi_{m})
- B \, ( S^{m}_{ji} \,\partial_k  \psi_{m} + \psi_{m} \, \partial_k S^{m}_{ji})
\\
& \quad
+ a_{15}\, S^{m}_{ik} \,( \partial_m\psi_{j} -  \partial_j\psi_{m})
- B \,( S^{m}_{ik} \, \partial_j\psi_{m} + \psi_{m} \, \partial_j S^{m}_{ik}) 
\\
& \quad
+ a_{16}\, S^{m}_{ki} \,( \partial_m\psi_{j} - \partial_j\psi_{m})
+ B\, ( S^{m}_{ki} \, \partial_j\psi_{m} + \psi_{m} \, \partial_j S^{m}_{ki})
\\
& \quad
+ a_{17}\, S^{m}_{jk} \, ( \partial_m\psi_{i} - \partial_i\psi_{m})
+ B\, (S^{m}_{jk} \, \partial_i\psi_{m}
+  \psi_{m} \, \partial_i S^{m}_{jk} )
\\
& \quad
+ a_{18}\, S^{m}_{kj} \, ( \partial_m\psi_{i} - \partial_i\psi_{m} )
- B \,( S^{m}_{kj} \, \partial_i\psi_{m} 
+  \psi_{m} \, \partial_i S^{m}_{kj})\,.
\end{align*}

Now, if we put $ B = 0$, we get a linear combination of operators from Lemma \ref{Lm: 4.2}.

Finally, putting $B = 1$ and the others free variables are vanishing, we obtain
\begin{align}  \label{Eq: 4.3}
\Phi_{ijk}
&  =
S^{m}_{ij} \, \partial_k\psi_{m} 
+ \psi_{m} \, \partial_k S^{m}_{ij}
- S^{m}_{ji} \,\partial_k  \psi_{m} - \psi_{m} \, \partial_k S^{m}_{ji}
\\
& \quad\nonumber
- S^{m}_{ik} \, \partial_j\psi_{m} 
- \psi_{m} \, \partial_j S^{m}_{ik}
+ S^{m}_{ki} \, \partial_j\psi_{m} 
+ \psi_{m} \, \partial_j S^{m}_{ki}
\\
& \quad\nonumber
+ S^{m}_{jk} \, \partial_i\psi_{m} 
+  \psi_{m} \, \partial_i S^{m}_{jk}
-  S^{m}_{kj} \, \partial_i\psi_{m} 
-  \psi_{m} \, \partial_i S^{m}_{kj} 
 \\
& \quad\nonumber
 + \psi_{m} \, \big(\partial_i S^{m}_{jk} 
 -  \partial_i S^{m}_{kj}
-  \partial_j S^{m}_{ik} 
+ \partial_j S^{m}_{ki}
+  \partial_k S^{m}_{ij}
-  \partial_k S^{m}_{ji}
\big)
\\
& = \nonumber
(S^{m}_{ij}  
- S^{m}_{ji} ) \,\partial_k  \psi_{m} 
+ (S^{m}_{ki}  
- S^{m}_{ik} )\, \partial_j\psi_{m} 
+ (S^{m}_{jk}  
-  S^{m}_{kj} )\, \partial_i\psi_{m}  
 \\
& \quad\nonumber
 + \psi_{m} \, \big(\partial_i S^{m}_{jk} 
 -  \partial_i S^{m}_{kj}
+ \partial_j S^{m}_{ki}
-  \partial_j S^{m}_{ik} 
+  \partial_k S^{m}_{ij}
-  \partial_k S^{m}_{ji}
\big)
 \,.
\end{align}
The operator $ \Phi(S,\psi) $ defined by \eqref{Eq: 4.3} is a multiple of
\begin{align*}
 d(\psi \circ \Alt S) 
\end{align*}
described in Lemma \ref{Lm: 4.4}. So all natural operators are linear combinations of 19 operators from
Lemmas \ref{Lm: 4.1} -- \ref{Lm: 4.4}.
\end{proof}

\subsection{The case of tangent-valued 2-forms}

Now, we assume that $ S $ is a tangent valued 2-form, i.e. $ \Alt S = S $.
Then there are 3 independent operators given by Lemma \ref{Lm: 4.1}, 3 independent operators given by Lemma \ref{Lm: 4.2} and 3 independent operators given by Lemma \ref{Lm: 4.3}. 

\begin{rmk}\label{Rm: 4.1}
{\rm
If $ S $ is a tangent-valued 2-form, then we have the Yano-Ako operator, \cite{YanAko68}, defined as 
\begin{align*}
\Phi(S,\psi)(X,Y,Z)
& =
(L_{S(X,Y)}\psi)(Z) - (L_{X}(\psi\circ S))(Z,Y)
\\
&\quad 
- (L_{Y}(\psi\circ S))(X,Z)
+ (\psi\circ S)([X,Y],Z)\,.
\end{align*}
We can express this operator as the linear combination of the basic operators in the form
\begin{align*}
\Phi(S,\psi)(X,Y,Z)
& =
d(\psi \circ S)(X,Z,Y) + d\psi(S(X,Y),Z)\,.
\end{align*}
}
\end{rmk}

\begin{rmk}\label{Rm: 4.2}
{\rm 
Let as assume that $ S $ is a tangent-valued 2-form.
According to \cite{FroNij56} and \cite[p. 69]{KolMicSlo93} we can define the Lie derivative of $ \psi $ with respect to $ S $ as
\begin{equation*}
L_{S}\psi = [i_{S},d]\psi = i_{S} d\psi + d i_{S}\psi\,
\end{equation*}
which is a 3-form. It is easy to see that
\begin{align*}
(L_{S}\psi)(X,Y,Z) 
& = d\psi(S(X,Y),Z) + 
d\psi(S(Y,Z),X) + d\psi(S(Z,X),Y)
\\
& \quad
+ d(\psi \circ S)(X,Y,Z)\,.
\end{align*}
}
\end{rmk}

\begin{rmk}\label{Rm: 4.3}
{\rm
If $ S $ is a tangent-valued 2-form, 
then we can consider the Lie derivation of $ \psi \circ S $ with respect to the identity tensor $ \mathbb{I} $
and obtain the tangent valued 3-form
\begin{align*}
L_{\mathbb{I}}(\psi \circ S) = i_{\mathbb{I}} d (\psi \circ S) - di_{\mathbb{I}} (\psi \circ S) = d(\psi \circ S)\,.
\end{align*}

The Yano-Ako operator is antisymmetric in the first two arguments. On the other hand $L_S\psi$ and $  L_{\mathbb{I}}(\psi \circ S) $ have values in 3-forms. If we assume the antisymmetrization of the Yano-Ako operator we get the following identity, \cite{YanAko68},
\begin{align*}
3 \Alt \Phi(\psi,S)
=
L_S \psi + 2\, L_{\mathbb{I}}(\psi \circ S)\,.
\end{align*}
}
\end{rmk}



\end{document}